\documentclass[a4paper,11pt,makeidx]{amsart}
\usepackage{amscd}
\usepackage{xypic}  
\usepackage{amssymb}
\usepackage{amsthm}
\usepackage{epsf}
\makeindex
\newtheoremstyle{break}
  {9pt}
  {9pt}
  {\itshape}
  {}
  {\bfseries}
  {.}
  {\newline}
  {}
\theoremstyle{break}
\newtheorem{bthm}{Theorem}

\newtheorem{bcor}{Corollary}
\theoremstyle{plain}
\newtheorem{thm}{Theorem}[section]
\newtheorem{cor}[thm]{Corollary}

\newtheorem{lemma}[thm]{Lemma}

\newtheorem{prop}[thm]{Proposition}

\newtheorem{defn}[thm]{Definition}
\newtheorem{notn}[thm]{Notation and conventions}

\newtheorem{rem}[thm]{Remark}

\renewcommand{\proofname}{Proof}

\def\RR{{\mathbb R}}

\def\Sing{\operatorname{Sing}}

\def\max{\operatorname{max}}

\def\Nlc{\operatorname{Nlc}}
\def\Nnef{\operatorname{Nnef}}
\def\Nna{\operatorname{Nna}}
\def\Supp{\operatorname{Supp}}
\def\Exc{\operatorname{Exc}}
\def\Null{\operatorname{Null}}

\def\ord{\operatorname{ord}}

\def\NN{{\mathbb N}}
\def\QQ{{\mathbb Q}}

\def\B{\mathbf{B}}

\def\D{{\Delta}}

\def\O{{\mathcal O}}

\def\*{\otimes}

\def\eqv{\equiv}

\def\+{\oplus}                   
\def\*{\otimes}                  

\def\Supp{\operatorname{Supp}}

\def\Bs{\operatorname{Bs}}

\begin{document}

\title[Nakamaye's theorem on  log canonical pairs]{Nakamaye's theorem on log canonical pairs} 

\dedicatory{\normalsize \dag \ Dedicated to the memory of Fassi} 

\author[S. Cacciola and A.F. Lopez]{Salvatore Cacciola* and Angelo Felice Lopez*}

\thanks{* Research partially supported by the MIUR national project ``Geometria delle variet\`a 
algebriche" PRIN 2010-2011.}

\address{\hskip -.43cm Dipartimento di Matematica e Fisica, Universit\`a di Roma 
Tre, Largo San Leonardo Murialdo 1, 00146, Roma, Italy. e-mail {\tt cacciola@mat.uniroma3.it, lopez@mat.uniroma3.it}}

\thanks{{\it 2010 Mathematics Subject Classification} : Primary 14C20, 14F18. Secondary 14E15, 14B05}

\begin{abstract} 
We generalize Nakamaye's description, via intersection theory, of the augmented base locus of a big and nef divisor on a normal pair with log-canonical singularities or, more generally, on a normal variety with non-lc locus of dimension $\leq 1$. We also generalize Ein-Lazarsfeld-Musta{\c{t}}{\u{a}}-Nakamaye-Popa's description, in terms of valuations, of the subvarieties of the restricted base locus of a big divisor on a normal pair with klt singularities. 
\end{abstract}

\maketitle

\section{Introduction}
\label{intro}

Let $X$ be a normal complex projective variety and let $D$ be a big $\QQ$-Cartier $\QQ$-divisor on $X$. The stable base locus
\[ \B(D)= \bigcap_{E \geq 0 : E \sim _{\QQ} D} {\rm Supp}(E) \]
is an important closed subset associated to $D$, but it is often difficult to handle. On the other hand, there are two, perhaps even more important, base loci associated to $D$. 

One of them is the augmented base locus (\cite{n}, \cite[Def. 1.2]{elmnp1}) 
\[ \B_+(D)= \bigcap_{E \geq 0 :  D - E \mbox{\ ample}} {\rm Supp}(E) \]
where $E$ is a $\QQ$-Cartier $\QQ$-divisor. Since this locus measures the failure of $D$ to be ample, it has proved to be a key tool in several recent important results
in birational geometry, such as Takayama \cite{t}, Hacon and McKernan's \cite{hm} effective birationality of pluricanonical maps or Birkar, Cascini, Hacon and McKernan's \cite{bchm} finite generation of the canonical ring, just to mention a few.

One way to compute $\B_+(D)$ is to pick a sufficiently small ample $\QQ$-Cartier $\QQ$-divisor $A$ on $X$, because then one knows that $\B_+(D) = \B(D-A)$ by  \cite[Prop. 1.5]{elmnp1}. 

In the case when $D$ is also nef, for every subvariety $V \subset X$ of dimension $d \geq 1$ such that $D^d \cdot V = 0$, we have that $D_{|V}$ is not big, whence $s(D-A)_{|V}$ cannot be effective for any $s \in \NN$ and therefore $V \subseteq \B(D-A) = \B_+(D)$. Now define
\[ {\rm Null}(D) = \bigcup_{V \subset X: D^d \cdot V = 0} V \]
so that, by what we just said,
\begin{equation}
\label{nak}
{\rm Null}(D) \subseteq \B_+(D).
\end{equation}
A somewhat surprising result of Nakamaye \cite[Thm. 0.3]{n} (see also \cite[\S 10.3]{laz}) asserts that, if $X$ is {\it smooth} and $D$ is big and nef, then in fact equality holds in \eqref{nak}.
 
As is well-known, in birational geometry, one must work with normal varieties with some kind of (controlled) singularities. In the light of this, it becomes apparent that it would be nice to have a generalization of Nakamaye's Theorem to normal varieties. While in positive characteristic the latter has been recently proved to hold, on any projective scheme, by Cascini, McKernan and Musta{\c{t}}{\u{a}} \cite[Thm. 1.1]{cmm}, we will show in this article a generalization to normal complex varieties with log canonical singularities. This partially answers a question in \cite{cmm}. 

More precisely let us define

\begin{defn}
{\rm Let $X$ be a normal projective variety. The {\bf non-lc locus} of $X$ is
\[ X_{\mathrm{nlc}} =\bigcap_{\D} \mathrm{Nlc}(X,\D) \]
where $\D$ runs among all effective Weil $\QQ$-divisors such that $K_X + \D$ is $\QQ$-Cartier and $\mathrm{Nlc}(X,\D)$ is the locus of points $x \in X$ such that $(X, \D)$ is not log canonical at $x$.}
\end{defn}

Using Ambro's and Fujino's theory of non-lc ideal sheaves \cite{a}, \cite{f} and a modification of some results of de Fernex and Hacon \cite{dh}, we prove

\vskip .3cm

\begin{bthm}
\label{main}
Let $X$ be a normal projective variety such that $\dim X_{\mathrm{nlc}} \leq 1$. 
Let $D$ be a big and nef $\QQ$-Cartier $\QQ$-divisor on $X$. Then 
\[ \B_+(D)=\Null(D). \]
\end{bthm}

\vskip .2cm

This easily gives the following

\begin{bcor}
\label{dim}
Let $X$ be a normal projective variety such that $\dim \Sing(X) \leq 1$ or $\dim X \leq 3$ or there exists an effective Weil $\QQ$-divisor $\D$ such that $(X, \D)$ is log canonical. 

Let $D$ be a big and nef $\QQ$-Cartier $\QQ$-divisor on $X$. Then 
\[ \B_+(D)=\Null(D). \]
\end{bcor}

\vskip .2cm

Moreover, using a striking result of Gibney, Keel and Morrison \cite[Thm. 0.9]{gkm}, we can give a very quick application to the moduli space of stable pointed curves.
 
\begin{bcor}
\label{mod}
Let $g \geq 1$ and let  $D$ be a big and nef  $\QQ$-divisor on $\overline{M}_{g,n}$. Then 
\[ \B_+(D) \subseteq \partial \overline{M}_{g,n}. \]
\end{bcor}
Thus, for example, one gets new compactifications of $M_{g,n}$ by taking rational maps associated to such divisors.
 
The second base locus associated to any pseudoeffective  $\RR$-Cartier $\RR$-divisor $D$, measuring how far $D$ is from being nef, is the restricted base locus \cite[Def. 1.12]{elmnp1}.

\begin{defn}
{\rm Let $X$ be a normal projective variety and let $D$ be a pseudoeffective $\RR$-Cartier $\RR$-divisor on $X$. The {\bf restricted base locus} of $D$ is
\[ \B_-(D)= \bigcup_{A {\ \rm ample}} \B(D+A) \]
where $A$ runs among all ample $\RR$-Cartier $\RR$-divisors such that $D+A$ is a $\QQ$-divisor.}
\end{defn}
Restricted base loci are countable unions of subvarieties by \cite[Prop. 1.19]{elmnp1}, but not always closed \cite[Thm. 1.1]{les}.

For a big $\QQ$-divisor $D$ on a {\it smooth} variety $X$, the subvarieties of $\B_-(D)$ are precisely described in \cite[Prop. 2.8]{elmnp1} (also in positive characteristic in \cite[Thm. 6.2]{mus}) in terms of asymptotic valuations. 

\begin{defn} 
\label{as1}
{\rm (\cite[Def. III.2.1]{nak}, \cite[Lemma 3.3]{elmnp1}, \cite[\S 1.3]{bbp}, \cite[\S 2]{dh})
Let $X$ be a normal projective variety, let  $D$ be an $\RR$-Cartier $\RR$-divisor on $X$ and let $v$ be a divisorial valuation on $X$, that is $v$ is a positive integer multiple of the valuation associated to a prime divisor $\Gamma$ lying on a birational model $f : Y \to X$. The center of $v$ on $X$ is $c_X(v)=f(\Gamma)$. 

If $D$ is big, we set
\[ v(\|D\|) = \inf \{v(E), E \ \mbox{effective} \ \RR\mbox{-Cartier} \ \RR\mbox{-divisor on X such that} \ E \eqv D\}; \]
if $D$ is pseudoeffective, we pick an ample divisor $A$ and set
\[ v(\|D\|) = \lim_{\varepsilon \to 0^+} v(\|D + \varepsilon A\|). \]
If  $D$ is a $\QQ$-Cartier $\QQ$-divisor such that $\kappa(D) \geq 0$ and $b \in \NN$ is such that $bD$ is Cartier and $|bD| \neq \emptyset$, we set
(see \cite[Def. 2.14]{cd} or \cite[Def. 2.2]{elmnp1} for the case $D$ big)
\[ v(\langle D \rangle) = \lim_{m \to + \infty}\frac{v(|mbD|)}{mb} \]
where, if $g$ is an equation, at the generic point of $c_X(v)$, of a general element in $|mbD|$, then $v(|mbD|) = v(g)$.}
\end{defn}

Now the main content of \cite[Prop. 2.8]{elmnp1} is that, given a discrete valuation $v$ on a {\it smooth} $X$ with center $c_X(v)$ and a big divisor $D$, then $c_X(v) \subseteq \B_-(D)$ if and only if $v(\|D\|) > 0$. Using the main result of \cite{cd} we give a generalization to normal pairs with klt singularities.

\begin{bthm}
\label{main2}
Let  $X$ be a normal projective variety such that there exists an effective Weil $\QQ$-divisor $\D$ with $(X,\D)$ a klt pair. Let $v$ be a divisorial valuation on $X$.
Then
\begin{itemize}
\item [(i)] If $D$ is a big Cartier divisor on $X$ we have 
\[ v(\langle D \rangle) > 0 \ \mbox{if and only if} \ c_X(v)\subseteq \B_-(D) \ \mbox{if and only if} \ \limsup_{m \to + \infty} v(|mD|) = + \infty. \]
\item [(ii)] If $D$ is a pseudoeffective $\RR$-Cartier $\RR$-divisor on $X$, we have
\[ v(\|D\|) > 0 \ \textrm{if and only if} \ c_X(v)\subseteq \B_-(D). \]
\end{itemize}
\end{bthm}

\vskip .2cm

\noindent {\it Acknowledgments}. We wish to thank Lorenzo Di Biagio for some helpful discussions.

\section{Non-lc ideal sheaves}
\label{basic}

\begin{notn}
{\rm Throughout the article we work over the complex numbers. Given a variety $X$ and a coherent sheaf of ideals $\mathcal{J} \subset  \mathcal{O}_X$, we denote by $\mathcal{Z}(\mathcal{J})$ the closed subscheme of $X$ defined by $\mathcal{J}$. If $X$ is a normal projective variety and $\D$ is a Weil $\QQ$-divisor on $X$,  we call $(X, \D)$ a {\bf pair} if $K_X + \D$ is $\QQ$-Cartier. We refer to \cite[Def. 2.34]{km} for the various notions of singularities of pairs.}
\end{notn} 

\begin{defn}
{\rm Let $X$ be a normal projective variety and let $\D = \sum\limits_{i=1}^s d_i D_i$ be a Weil $\QQ$-divisor on $X$, where the ${D_i}'s$ are distinct prime divisors. 

Given $a \in \RR$ we set $\D^{>a}=\sum\limits_{1 \leq i \leq s : d_i>a} d_i D_i$, $\D^+=\D^{>0}$, $\D^- =(-\D)^+$ and $\D^{<a}=-((-\D)^{>-a})$. The {\bf round up} of $\D$ is $\lceil \D \rceil = \sum\limits_{i=1}^s \lceil d_i \rceil D_i$ and the {\bf round down} is $\lfloor \D \rfloor = \sum\limits_{i=1}^s \lfloor d_i \rfloor D_i$. We also set $\D^{\#} =  \D^{<-1} + \D^{>-1}$.}
\end{defn} 

The following is easily proved.

\begin{rem}
\label{sha}
{\rm Let $X$ be a normal projective variety and let $\D, \D'$ be Weil $\QQ$-divisors on $X$. Then
\begin{itemize}
\item[(i)] $\lceil  (-\D)^{\#} \rceil = \lceil -(\D^{<1}) \rceil - \lfloor \D^{>1} \rfloor$; 
\item[(ii)] If $\D \leq \D'$, then $\lceil \D^{\#} \rceil \leq \lceil  (\D')^{\#} \rceil$.
\end{itemize}}
\end{rem}

We recall the definition of non-lc ideal sheaves \cite[Def. 4.1]{a}, \cite[Def. 2.1]{f}.

\begin{defn}
\label{inlc}
{\rm Let $(X, \D)$ be a pair and let $f: Y \to X$ be a resolution  of $X$ such that $\D_Y :=f^{\ast}(K_X+\D)-K_Y$ has simple normal crossing support.
The {\bf non-lc ideal sheaf associated to} $(X,\D)$ is
\[ \mathcal{J}_{NLC}(X,\D)=f_{\ast}\mathcal{O}_Y(\lceil - (\D_Y^{<1}) \rceil -  \lfloor \D_Y^{>1}  \rfloor). \]}
\end{defn}

\begin{rem}
\label{nlc}
{\rm Non-lc ideal sheaves are well-defined by \cite[Prop. 2.6]{f}, \cite[Rmk. 4.2(iv)]{a}. Moreover, when $\D$ is effective and $f : Y \to X$ is a log-resolution of $(X, \D)$, we have that the non-lc locus of $(X, \D)$ is, set-theoretically, $\Nlc(X,\D)= f(\Supp(\D_Y^{>1})) = \mathcal{Z}(\mathcal{J}_{NLC}(X,\D))$ \cite[Lemma 2.2]{f}.}
\end{rem}

\begin{rem}
\label{ic}
{\rm The non-lc ideal sheaf of a pair $(X,\D)$ with $\D$ effective is an integrally closed ideal.}
\begin{proof}
With notation as in Definition \ref{inlc}, set $G =  \lceil - (\D_Y^{<1})  \rceil$ and $N=  \lfloor \D_Y^{>1}  \rfloor$, so that $G$ and $N$ are effective divisors without common components, $G$ is $f$-exceptional and $\mathcal{J}_{NLC}(X,\D)=f_{\ast}\mathcal{O}_Y(G-N) = f_{\ast}\mathcal{O}_Y(-N)$ by Fujita's lemma \cite[Lemma 2.2]{fta}, \cite[Lemma 1-3-2]{kmm}, \cite[Lemma 4.5]{dh}. Therefore $\mathcal{J}_{NLC}(X,\D)$ is an ideal sheaf and it is integrally closed by \cite[Prop. 9.6.11]{laz}.
\end{proof}
\end{rem}

Our next goal is to prove, using techniques and results in de Fernex-Hacon \cite{dh}, that non-lc ideal sheaves have a unique maximal element. To this end we will use some results of Fujino \cite{f} and de Fernex-Hacon \cite{dh} that we wish to recall for the reader's convenience.

\begin{lemma} \cite[Lemma 2.7]{f} 
\label{fuj}
Let $g : Y' \to Y$ be a proper birational morphism between smooth varieties and let $B_Y$ be an $\RR$-divisor on $Y$ having simple normal crossing support. Assume that  $B_{Y'} := g^{\ast}(K_Y+B_Y)-K_{Y'}$ also has simple normal crossing support. Then 
\[ g_{\ast} \O_{Y'} (\lceil -(B_{Y'}^{<1}) \rceil - \lfloor B_{Y'}^{>1} \rfloor ) \cong \O_Y (\lceil -(B_Y^{<1}) \rceil - \lfloor B_Y^{>1} \rfloor ). \]
\end{lemma}

\begin{defn} 
\label{dfh1}
{\rm (\cite[Def. 3.1]{dh})
Let $f: Y \to X$ be a proper birational morphism between normal varieties and let  $K_Y$ be a canonical divisor on $Y$ and $K_X = f_{\ast} K_Y$. For every $m \geq 1$ define $K_{m, Y/X} = K_Y - \frac{1}{m}f^{\natural}(mK_X)$, where $f^{\natural}(mK_X)$ is the divisor on $Y$ such that $\O_Y(-f^{\natural}(mK_X)) = (\O_X(-mK_X) \cdot \O_Y)^{\vee \vee}$.}
\end{defn}

\begin{lemma}
\label{dh2}
Let $m \geq 1$. In (i)-(iv) below let $f: Y \to X$ be a proper birational morphism between normal varieties. Then
\begin{itemize}
\item [(i)]  If $X$ is Gorenstein then $K_{m, Y/X} = K_{Y/X} := K_Y + f^{\ast}(-K_X)$;
\item [(ii)]  \cite[Rmk 3.3]{dh} For all $q \geq 1$ we have $K_{m, Y/X} \leq K_{mq, Y/X}$;
\item [(iii)] \cite[Lemma 3.5]{dh} Assume that $mK_Y$ is Cartier and $\O_X(-mK_X) \cdot \O_Y$ is invertible. Let $Y'$ be a normal variety and let $g : Y' \to Y$ be a proper birational morphism. Then $K_{m, Y'/X} = K_{m, Y'/Y} + g^{\ast} K_{m, Y/X}$;
\item [(iv)] \cite[Rmk 3.9]{dh} Let $(X,\D)$ be a pair with $\D$ effective and assume that  $m(K_X + \Delta)$ is Cartier. Then $K_Y + f_{\ast}^{-1}(\D) - f^{\ast}(K_X+\D) \leq K_{m, Y/X}$;
\item [(v)] \cite[Thm. 5.4 and its proof]{dh} For every $m \geq 2$ there exist a log-resolution $f: Y \to X$ of $(X, \mathcal{O}_X(-mK_X))$ and a Weil $\QQ$-divisor $\D_m$ on $X$ such that $m\D_m$ is integral, $\lfloor \D_m \rfloor = 0$, $(\D_m)_Y$ has simple normal crossing support, $f$ is a log-resolution for the log-pair $((X, \D_m), \mathcal{O}_X(-mK_X))$, $K_X + \D_m$ is $\QQ$-Cartier and $K_{m, Y/X} = K_Y + f_{\ast}^{-1}(\D_m) - f^{\ast}(K_X+\D_m)$.
\end{itemize}
\end{lemma}
In (iv) and (v) above $f_{\ast}^{-1}(\D)$ is the proper transform of $\D$. Note that our $\D_Y$ (Definition \ref{inlc}) is different from the one in \cite[Def. 3.8]{dh}.

Now we have

\begin{prop} 
\label{max}
Let $X$ be a normal projective variety. Then there exists a Weil $\QQ$-divisor $\D_0$ on $X$ such that $\lfloor \D_0 \rfloor = 0$,
$K_X + \D_0$ is $\QQ$-Cartier and 
\[ \mathcal{J}_{NLC}(X,\D) \subseteq \mathcal{J}_{NLC}(X,\D_0) \]
for every pair $(X,\D)$ with $\D$ effective.
\end{prop}
\begin{proof}
Fix a canonical divisor $K_X$ on $X$ and an integer $m \geq 2$. By Lemma \ref{dh2}(v) there exist a log-resolution $f: Y \to X$ of $(X, \mathcal{O}_X(-mK_X))$ and a Weil $\QQ$-divisor $\D_m$ on $X$ with the properties in (v). In particular $K_{m, Y/X}$ is $f$-exceptional. Now set
\[ \mathfrak{a}_m(X) = f_{\ast}\mathcal{O}_Y( \lceil (K_{m, Y/X})^{\#} \rceil).\]
As in the proof of Remark \ref{ic} we get that $\mathfrak{a}_m(X)$ is a coherent ideal sheaf. Let us check that its definition is independent of the choice of $f$. Let $f': Y' \to X$ be another log-resolution of $(X, \mathcal{O}_X(-mK_X))$ and assume, as we may, that $f'$ factors through $f$ and a morphism $g : Y' \to Y$. By Lemma \ref{dh2}(iii) and (i) we have
$K_{m, Y'/X} = K_{Y'/Y} + g^{\ast} K_{m, Y/X} = K_{Y'} - g^{\ast} (K_Y - K_{m, Y/X})$, whence
\begin{equation}
\label{ben}
(fg)_{\ast}\mathcal{O}_{Y'}(\lceil (K_{m, Y'/X})^{\#} \rceil) = f_{\ast}(g_{\ast} \mathcal{O}_{Y'} (\lceil (K_{Y'} - g^{\ast} (K_Y - K_{m, Y/X}))^{\#} \rceil)).
\end{equation}
Now set $B_Y = - K_{m, Y/X}$ and $B_{Y'} = g^{\ast}(K_Y+B_Y) - K_{Y'}$ so that, using Remark \ref{sha}(i) and Lemma \ref{fuj}, we have
\[ g_{\ast} \mathcal{O}_{Y'} (\lceil (K_{Y'} - g^{\ast} (K_Y - K_{m, Y/X}))^{\#} \rceil) = g_{\ast} \mathcal{O}_{Y'} (\lceil (-B_{Y'})^{\#} \rceil) = g_{\ast} \mathcal{O}_{Y'} (\lceil -(B_{Y'}^{<1}) \rceil - \lfloor B_{Y'}^{>1} \rfloor ) = \]
\[ = \mathcal{O}_Y (\lceil -(B_Y^{<1}) \rceil - \lfloor B_Y^{>1} \rfloor ) = \mathcal{O}_Y (\lceil (-B_Y)^{\#} \rceil) = \mathcal{O}_Y( \lceil (K_{m, Y/X})^{\#} \rceil) \]

and by \eqref{ben} we get
\[ (fg)_{\ast}\mathcal{O}_{Y'}(\lceil (K_{m, Y'/X})^{\#} \rceil) = f_{\ast} \mathcal{O}_Y( \lceil (K_{m, Y/X})^{\#} \rceil) \]
that is $\mathfrak{a}_m(X)$ is well defined.

We now claim that the set $\{\mathfrak{a}_m(X), m \geq 2 \}$ has a unique maximal element. In fact, given $m, q \geq 2$, let $f: Y \to X$ be a log-resolution of $(X, \mathcal{O}_X(-mK_X)) + \mathcal{O}_X(-mqK_X))$. By Lemma \ref{dh2}(ii) and Remark \ref{sha}(ii) we have $\lceil (K_{m, Y/X})^{\#} \rceil \leq \lceil (K_{mq, Y/X})^{\#} \rceil$ and therefore $\mathfrak{a}_m(X) \subseteq \mathfrak{a}_{mq}(X)$. Using the ascending chain condition on ideals we conclude that $\{\mathfrak{a}_m(X), m \geq 2 \}$ has a unique maximal element, which we will denote by $\mathfrak{a}_{\rm max}(X)$. 

Next let us show that all the ideal sheaves $\mathfrak{a}_m(X)$, for $m \geq 2$ (whence in particular also $\mathfrak{a}_{\rm max}(X)$), are in fact non-lc ideal sheaves of a suitable pair. 

Let  $\D_m$ be as above, so that, by Remark \ref{sha}(i) and using $\lceil  (-f_{\ast}^{-1}(\D_m))^{\#} \rceil = 0$, we have
\[ \lceil - ((\D_m)_Y^{<1}) \rceil -  \lfloor (\D_m)_Y^{>1}  \rfloor = \lceil  (-(\D_m)_Y)^{\#} \rceil = \lceil  (K_{m, Y/X} - f_{\ast}^{-1}(\D_m))^{\#} \rceil = \]
\[  = \lceil  (K_{m, Y/X})^{\#}  \rceil + \lceil (- f_{\ast}^{-1}(\D_m))^{\#} \rceil = \lceil  (K_{m, Y/X})^{\#} \rceil \]
whence
\[ \mathcal{J}_{NLC}(X,\D_m)=f_{\ast}\mathcal{O}_Y(\lceil - ((\D_m)_Y^{<1}) \rceil -  \lfloor (\D_m)_Y^{>1}  \rfloor) = f_{\ast}\mathcal{O}_Y( \lceil (K_{m, Y/X})^{\#} \rceil) = \mathfrak{a}_m(X). \]
To finish the proof, let $(X,\D)$ be a pair with $\D$ effective and let $q \in \NN$ be such that $q(K_X + \D)$ is Cartier. Let $m_0 \geq 2$ be such that $\mathfrak{a}_{\rm max}(X) = \mathfrak{a}_{m_0}(X) = \mathfrak{a}_{qm_0}(X)$. By what we proved above, there exists $\D_0 := \D_{qm_0}$ such that $\mathcal{J}_{NLC}(X,\D_0) =  \mathfrak{a}_{\rm max}(X)$. By Lemma \ref{dh2}(iv)  we have that  $- \D_Y \leq K_Y + f_{\ast}^{-1}(\D) - f^{\ast}(K_X+\D) \leq K_{qm_0, Y/X}$, whence also, by Remark \ref{sha} (i) and (ii),
\[ \lceil - (\D_Y^{<1}) \rceil -  \lfloor \D_Y^{>1}  \rfloor = \lceil (- \D_Y)^{\#} \rceil \leq \lceil (K_{qm_0, Y/X})^{\#} \rceil \]
and therefore
\[ \mathcal{J}_{NLC}(X,\D)=f_{\ast}\mathcal{O}_Y(\lceil - (\D_Y^{<1}) \rceil -  \lfloor \D_Y^{>1}  \rfloor) \subseteq f_{\ast}\mathcal{O}_Y(\lceil (K_{qm_0, Y/X})^{\#} \rceil) = \mathfrak{a}_{\rm max}(X) = \mathcal{J}_{NLC}(X,\D_0). \]
\end{proof}

\section{Proof of Theorem \ref{main}}
\label{b+}

We record the following lemma, which is also of independent interest.

\begin{lemma} 
\label{bs}
Let $(X,\D)$ be a pair with $\D$ effective and let $D$ be an effective Cartier divisor on $X$. Then there exists $c = c(X,\D, D) \in \NN$ such that the set-theoretic equality
\[ \Bs|D| \cup \Nlc(X,\D) = \mathcal{Z}(\mathcal{J}_{NLC}(X,\D+E_1+\dots+E_c)) \]
holds for some $E_1,\dots, E_c \in |D|$.
\end{lemma}

\begin{proof}
Let $f:Y \to X$ be a log-resolution of $(X,\D)$ and of the linear series $|D|$ such that $f_{\ast}^{-1} \D + \Bs|f^{\ast}D| + \Exc(f)$ has simple normal crossing support.
Write $\D_Y = \D_Y^+ - \D_Y^-$, where $\D_Y^+$ and $\D_Y^-$ are effective simple normal crossing support $\QQ$-divisors without common components. Then $\D_Y^-= \sum_{i=1}^s \delta_i D_i$, for some non-negative $\delta_i \in \QQ$ and distinct prime divisors $D_i$'s and define 
\[ c =  \lceil \max\{\delta_i, 1 \leq i \leq s \}  \rceil +2. \]
Moreover we have that $|f^{\ast}D|=|M|+F$, where $|M|$ is base-point free and $\Supp(F)= \Bs|f^{\ast}D|$. By Bertini's Theorem and \cite[Lemma 9.1.9]{laz}, we can choose $M_1,\dots, M_c \in |M|$ general divisors such that, for all $j=1,\dots,c$, $M_j$ is smooth, every component of $M_j$ is not a component of $\D_Y, M_1,\dots, M_{j-1}$ and $\D_Y+M_1+\dots+M_c+F$ has simple normal crossing support. Now, for all $j = 1, \ldots, c$, $M_j+F \in |f^{\ast}D|$, so that there exists $E_j \in |D|$ such that $M_j+F =f^{\ast}E_j$. Set $E=E_1+\dots +E_c$ and notice that $f$ is also a log-resolution of $(X,\D+E)$.

By Remark \ref{nlc} we have $\Nlc(X,\D)=\mathcal{Z}(\mathcal{J}_{NLC}(X,\D))\subseteq \mathcal{Z}(\mathcal{J}_{NLC}(X,\D+E))$, the latter inclusion following by Remark \ref{sha}(i) and (ii), because $E$ is effective. Also, for every prime divisor $\Gamma$ in the support of $F$ we get for the discrepancies
\[ a(\Gamma,X,\D+E)=a(\Gamma,X,\D)-\ord_{\Gamma}(f^{\ast} E) =-\ord_{\Gamma} (\D_Y) - \ord_{\Gamma}(f^{\ast} E) \leq \]
\[ \leq \ord_{\Gamma} (\D_Y^-) - \ord_{\Gamma}(f^{\ast} E) \leq  \max\{\delta_i, 1 \leq i \leq s \} -\ord_{\Gamma}(M_1+\dots+M_c+cF)\leq -2 \]
whence $f(\Gamma) \subseteq \Nlc(X, \D+E)$. As $\Bs|D|$ is the union of such $f(\Gamma)$'s, using Remark \ref{nlc}, we get the inclusion $\Bs |D| \subseteq \Nlc(X, \D+E) = \mathcal{Z}(\mathcal{J}_{NLC}(X,\D+E))$.
 
On the other hand notice that $(\D+E)_Y =f^{\ast}(K_X+\D+E)-K_Y= \D_Y +f^{\ast}E$. Also $\D_Y + f^{\ast} E = \D_Y + M_1+\dots+M_c+cF$, so that
\[ \Supp((\D+E)_Y^{>1}) = \Supp((\D_Y + f^{\ast} E)^{>1}) \subseteq \Supp(F) \cup \Supp (\D_Y^{>1}) \]
whence
\[ f(\Supp((\D+E)_Y^{>1})) \subseteq f(\Supp(F)) \cup f(\Supp (\D_Y^{>1})) = \Bs|D| \cup \Nlc(X,\D). \]
Therefore, by Remark \ref{nlc}, 
\[ \mathcal{Z}(\mathcal{J}_{NLC}(X,\D+E)) = \Nlc(X, \D+E) = f(\Supp((\D+E)_Y^{>1})) \subseteq \Bs|D| \cup \Nlc(X,\D). \]
\end{proof}

Now we essentially follow the proof of Nakamaye's Theorem as in \cite[\S 10.3]{laz} and \cite[Thm. 0.3]{n}.
 
\renewcommand{\proofname}{Proof of Theorem {\rm \ref{main}}}  
\begin{proof}
We can assume that $D$ is a Cartier divisor. The issue is of course to prove that $\B_+(D) \subseteq \Null(D)$, since the opposite inclusion holds on any normal projective variety, as explained in the introduction.

By Proposition \ref{max} and Remark \ref{nlc} there is an effective Weil $\QQ$-divisor $\D$ on $X$ such that $K_X + \D$ is $\QQ$-Cartier and $\mathrm{Nlc}(X,\D)= X_{\mathrm{nlc}}$, so that $\dim \mathrm{Nlc}(X,\D) \leq 1$.

Let $A$ be an ample Cartier divisor such that $A-(K_X+\D)$ is ample. As in \cite[Proof of Thm. 10.3.5]{laz}) we can choose $a, p \in \NN$ sufficiently large such that
\[ \B_+(D)= \B(aD-2A)= \Bs|paD-2pA|. \]
By Lemma \ref{bs} there exist $c \in \NN$ and a Cartier divisor $E$ on $X$ such that
\[ \B_+(D)\cup \Nlc(X,\D) = \mathcal{Z}(\mathcal{J}_{NLC}(X,\D+E)) \]
and $E \equiv c(paD-2pA)=qaD-2qA$, where $q:= cp \in \NN$. 

Set $Z =  \mathcal{Z}(\mathcal{J}_{NLC}(X,\D+E))$. For $m \geq qa$, we get that 
\[ mD-qA-(K_X+\D+E)\equiv (m-qa)D+qA-(K_X+\D) \] 
is ample, whence $H^1(X, \mathcal{J}_{NLC}(X,\D+E) \otimes \mathcal{O}_X(mD-qA))=0$, for $m\geq qa$
by  \cite[Thm. 3.2]{f},  \cite[Thm. 4.4]{a}, so that the restriction map
\begin{equation} 
\label{uno}
H^0(X,\mathcal{O}_X(mD-qA))\to H^0(Z, \mathcal{O}_Z(mD-qA)) \mbox{\ is surjective for} \ m \geq qa.
\end{equation} 
By contradiction let us assume that there exists an irreducible component $V$ of $\B_+(D)$, such that $V \not \subseteq \Null(D)$.
Now $V \subseteq \B_+(D) \subseteq \B(D-\frac{q}{m}A)\subseteq \Bs |mD-qA|$ for $m \in \NN$, whence the restriction map
\[ H^0(X,\mathcal{O}_X(mD-qA))\to H^0(V,\mathcal{O}_V(mD-qA)) \mbox{\ is zero for} \ m  \in \NN \]
and therefore, by \eqref{uno}, also
\begin{equation} 
\label{due}
H^0(Z,\mathcal{O}_Z(mD-qA))\to H^0(V,\mathcal{O}_V(mD-qA)) \mbox{\ is zero for} \ m \geq qa.
\end{equation} 
On the other hand $\dim V\geq 1$, as $\B_+(D)$ does not contain isolated points by \cite[Proposition 1.1]{elmnp2}(which holds on $X$ normal).
As $\dim \Nlc(X,\D) \leq 1$, this implies that $V$ is an irreducible component of $Z$. Moreover, as $V\not \subseteq \Null(D)$, we have that $D_{|_V}$ is big.

Now,  by Remark \ref{ic},  $\mathcal{J}_{NLC}(X,\D+E)$ is integrally closed, and exactly as in  \cite[Proof of Thm. 10.3.5]{laz} (the proof of this part holds on any normal projective variety) it follows that, for $m \gg 0$, $H^0(Z,\mathcal{O}_Z(mD-qA))\to H^0(V,\mathcal{O}_V(mD-qA))$ is not zero, thus contradicting \eqref{due}.
This concludes the proof.
\end{proof}
\renewcommand{\proofname}{Proof}

\renewcommand{\proofname}{Proof of Corollary {\rm \ref{dim}}}  
\begin{proof}
Note that, on any normal projective variety $X$, we have $X_{\mathrm{nlc}} \subseteq \Sing(X)$ (see for example \cite[Rmk 4.8]{cd}) and  if $\dim X\leq 3$, then $\dim \Sing(X) \leq 1$. Then just apply Theorem \ref{main}.
\end{proof}
\renewcommand{\proofname}{Proof}

\renewcommand{\proofname}{Proof of Corollary {\rm \ref{mod}}}
\begin{proof}
By \cite[Thm. 0.9]{gkm} we know that $\Null(D) \subseteq \partial \overline{M}_{g,n}$. On the other hand it is well-known
(see for example \cite[Lemma 10.1]{bchm}) that $(\overline{M}_{g,n}, 0)$ is klt, whence the conclusion follows by Theorem \ref{main}.
\end{proof}
\renewcommand{\proofname}{Proof}

\section{Restricted base loci on klt pairs}
\label{rests}

We first recall that, associated to a pseudoeffective divisor $D$, there are two more loci,  one that also measures how far $D$ is from being nef and another one that measures how far $D$ is from being nef and abundant.

\begin{defn}
\label{as2}
{\rm Let $X$ be a normal projective variety and let  $D$ be a pseudoeffective $\RR$-Cartier $\RR$-divisor on $X$. As in \cite[Def. 1.7]{bbp}, we define the {\bf non-nef locus}
\[ \Nnef(D) = \bigcup_{v : v(\|D\|) > 0} c_X(v) \]
where $v$ runs among all divisorial valuations on $X$, $c_X(v)$ is its center and  $v(\|D\|)$ is as in Definition \ref{as1}.

Let $D$ be a $\QQ$-Cartier $\QQ$-divisor such that $\kappa(D) \geq 0$. As in \cite[Def. 2.18]{cd}, we define the {\bf non nef-abundant locus}
\[ \Nna(D) = \bigcup_{v : v(\langle D \rangle) > 0} c_X(v) \]
where again $v$ runs among all divisorial valuations on $X$ and  $v(\langle D \rangle)$ is as in Definition \ref{as1}.}
\end{defn}

In the sequel we will use the fact that, for $D$ big (\cite[Lemma 3.3]{elmnp1}) or even abundant (\cite[Prop. 6.4]{leh}), we have $v(\|D\|) = v(\langle D \rangle)$, while in general they are different when $D$ is only pseudoeffective (\cite[Rmk 2.16]{cd}).

We will also use (see \cite[page 2]{bfj} and references therein) 

\noindent {\bf Izumi's Theorem}
\label{iz}
{\it Let $X$ be a normal variety over an algebraically closed field $k$ and let $0 \in X$ be a closed point. Let $m_0$ be the maximal ideal of the local ring $\O_{X,0}$ and set, for any $f  \in \O_{X,0}$, $\ord_0 (f) = \max\{j \geq 0 :  f \in m_0^j \}$.
For any divisorial valuation $v$ of $k(X)$ centered at $0$, there exists a constant $C = C(v) > 0$ such that
\[ C^{-1} \ord_0(f) \leq v(f) \leq C \ord_0(f).\]}
We start by proving an analogue of \cite[Prop. 2.8]{elmnp1} for $\Nna(D)$.

\begin{thm}
\label{nna}
Let $X$ be a normal projective variety, let $D$ be a $\QQ$-Cartier $\QQ$-divisor such that $\kappa(D) \geq 0$ and let $v$ be a divisorial valuation on $X$.
Then  
\[ c_X(v)\subseteq \Nna(D) \ \textrm{if and only if} \ v(\langle D \rangle)>0. \]
\end{thm}
\begin{proof}
We can assume that $D$ is Cartier and effective. By definition of $\Nna(D)$, we just need to prove that if $c_X(v) \subseteq \Nna(D)$, then $v(\langle D \rangle)>0$. 

We first prove the theorem when $X$ is smooth.  For any $p \in \NN$ let $b(|pD|)$ be the base ideal of $|pD|$, $\mathcal{J}(X,\|pD\|))$ the asymptotic multiplier ideal and denote by $b_p$ and $j_p$ the corresponding images in $R_v$, the DVR associated to $v$. As in \cite[\S 2]{elmnp1}, we get
\begin{equation}
\label{j}
v(\langle D \rangle) = \lim_{p \to +\infty} \frac{v(b_p)}{p} \geq \lim_{p \to +\infty} \frac{v(j_p)}{p} = \sup_{p \in \NN} \{\frac{v(j_p)}{p}\}.
\end{equation}
By \cite[Cor. 5.2]{cd} we have the set-theoretic equality
\[ \Nna(D)=\bigcup_{p \in \NN} \mathcal{Z}(\mathcal{J}(X,\|pD\|)) \]
whence there exists $p_0 \in \NN$ such that $c_X(v) \subseteq \mathcal{Z}(\mathcal{J}(X,\|p_0D\|))$, so that $v(j_{p_0})>0$ and \eqref{j} gives that $v(\langle D \rangle)>0$. 

We now prove the theorem for a divisorial valuation $\nu$ on $X$ such that $c_X(\nu) = \{x\}$ is a closed point.

As $c_X(\nu) \subseteq \Nna(D)$, there exists a divisorial valuation $v_0$ on $X$ such that $v_0(\langle D \rangle)>0$ and $x \in c_X(v_0)$.
Let $E_0$ be a prime divisor over $X$ such that  $v_0 = k \ord_{E_0}$ for some $k \in \NN$. We can assume that there is a birational morphism  $\mu:Y\to X$ from a smooth 
variety $Y$ such that $E_0 \subset Y$. As $\mu(E_0)=c_X(\ord_{E_0}) = c_X(v_0 )$, there is a point $y \in E_0$ such that $\mu(y) = x$. Let $\pi : Y' \to Y$ be the blow-up of $Y$ on $y$ with exceptional divisor $E_y$. For any $m \in \NN$ and $G \in |mD|$ we have
$$\ord_{E_y}(G) =\ord_{E_y}(\pi^*(\mu^*G)) = \ord_{y}(\mu^*G) \geq \ord_{E_0}(\mu^*G) = \ord_{E_0}(G)$$
therefore $\ord_{E_y}(\langle D \rangle) \geq \ord_{E_0}(\langle D \rangle) = \frac{1}{k} v_0(\langle D \rangle)>0$.
Since $c_X(\ord_{E_y})=\{x\}$, by Izumi's Theorem applied twice, there exist $C > 0, C' > 0$ such that for all $m \in \NN$ and $G \in |mD|$ we have
$\ord_{E_y}(G) \leq C' \ord_{x}(G) \leq C \nu(G)$.  Hence $\nu(\langle D \rangle) \geq \frac{1}{C} \ord_{E_y}(\langle D \rangle)>0$.

Finally let $v$ be any divisorial valuation on $X$ with $c_X(v)\subseteq \Nna(D)$. As above there is a birational morphism $f : Z \to X$ from a smooth 
variety $Z$ and a prime divisor $E \subset Z$ such that $v = h \ord_E$ for some $h \in \NN$. For every closed point $z \in E$ we have that  $\nu:= \ord_z$ is a divisorial valuation with $c_X(\nu) \subseteq c_X(\ord_E) \subseteq \Nna(D)$ and $c_X(\nu)$ is a closed point. Thus, by what we proved above, we have that $\ord_z(\langle f^*(D) \rangle) = \ord_z(\langle D \rangle)>0$ for all $z \in E$, so that $E \subseteq \Nna(f^*(D))$. As $Z$ is smooth, we get $v(\langle D \rangle) = h\ord_E(\langle D \rangle) = h \ord_E(\langle f^*(D) \rangle)>0$.
\end{proof}

We next prove an analogous result for $\Nnef(D)$.

\begin{thm}
\label{nnef}
Let $X$ be a normal projective variety, let $D$ be a pseudoeffective $\RR$-Cartier $\RR$-divisor on $X$ and let $v$ be a divisorial valuation on $X$.
Then
$$c_X(v) \subseteq \Nnef(D) \ \textrm{if and only if} \ v(\|D\|)>0.$$
\end{thm}

\begin{proof} Again we need to prove that $v(\|D\|) > 0$ if $c_X(v) \subseteq \Nnef(D)$.
By \cite[Lemmas 2.13 and 2.12]{cd}, there exists a sequence of ample $\RR$-Cartier $\RR$-divisors $\{A_m\}_{m \in \NN}$
such that $\|A_m\| \to 0$, $D+A_m$ is a big $\QQ$-Cartier $\QQ$-divisor for all $m \in \NN$ and
$$\Nnef(D)=\bigcup_{m\in \NN} \Nnef(D+A_m).$$
Then there is $m_0 \in \NN$ such that $c_X(v)\subseteq \Nnef(D+A_{m_0})$. As $D+A_{m_0}$ is big, we have $\Nnef(D+A_{m_0}) = \Nna(D+A_{m_0})$, whence $v(\|D+A_{m_0}\|) = v(\langle D+A_{m_0}\rangle) >0$ by Theorem \ref{nna}.
Therefore $0 < v(\|D+A_{m_0}\|) \leq v(\|D\|) + v(\|A_{m_0}\|) = v(\|D\|)$.
\end{proof}

\begin{rem}
{\rm Note that, given a normal projective variety $X$, Theorems {\rm \ref{nna}} and {\rm \ref{nnef}} can be rewritten as follows (where $x$ is a closed point).

If $D$ is a $\QQ$-Cartier $\QQ$-divisor on $X$ such that $\kappa(D)\geq 0$, then
$$\Nna(D)=\bigcup_{x \in X}  \{x\;|\;\{x\}=c_X(v) \mbox{ for some divisorial valuation } v \mbox{ with } v(\langle D\rangle)>0\}.$$

If $D$ is a pseudoeffective $\RR$-Cartier $\RR$-divisor on $X$, then
$$\Nnef(D)=\bigcup_{x \in X}  \{x\;|\;\{x\}=c_X(v) \mbox{ for some divisorial valuation } v \mbox{ with } v(\|D\|)>0\}.$$}
\end{rem}

Next we will prove Theorem {\rm \ref{main2}. We will use a singular version (see for example \cite[Def. 2.2]{cd}) of standard asymptotic multiplier ideal sheaves \cite[Ch. 11]{laz}.
\renewcommand{\proofname}{Proof of Theorem {\rm \ref{main2}}}
\begin{proof}
In both cases we have that $\Nnef(D) = \B_-(D)$ by \cite[Thm. 1.2]{cd}, whence also $\Nna(D) = \B_-(D)$ in case (i). Then (ii) follows by Theorem \ref{nnef} and the first equivalence in (i) by Theorem \ref{nna}. To complete the proof of (i) we need to show that if $\limsup_{m \to + \infty} v(|mD|) = + \infty$ then $v(\langle D\rangle) > 0$, the reverse implication being obvious. We will proceed similarly to \cite[Proof of Prop. 2.8]{elmnp1} and \cite[Proof of Lemma 4.1]{cd}. If $v(\langle D\rangle) = 0$, by what we just proved, we have that $c_X(v) \not\subseteq \B_-(D)$ and, by \cite[Cor. 5.2]{cd}, we have the set-theoretic equality
\[ \B_-(D) = \bigcup_{p \in \NN} \mathcal{Z}(\mathcal{J}((X,\D);\|pD\|)) \]
where $\mathcal{J}((X,\D);\|pD\|)$ is as in \cite[Def. 2.2]{cd}. Therefore $c_X(v) \not\subseteq  \mathcal{Z}(\mathcal{J}((X,\D); \|pD\|))$ for any $p \in \NN$. Let $H$ be a very ample Cartier divisor such that $H - (K_X + \D)$ is ample and let $n = \dim X$. By Nadel's vanishing theorem \cite[Thm. 9.4.17]{laz}, we deduce that $\mathcal{J}((X,\D);\|pD\|) \otimes \mathcal{O}_X((n + 1)H + pD)$ is $0$-regular in the sense of Castelnuovo-Mumford, whence globally generated, for every $p \in \NN$, and therefore $c_X(v) \not\subseteq \Bs|(n+1)H + pD|$. On the other hand, as $D$ is big, there is $m_0 \in \NN$ such that $m_0 D \sim (n + 1)H + E$ for some effective Cartier divisor $E$. Hence, for any $m \geq m_0$, we get $v(|mD|) = v(|(m-m_0)D + (n + 1)H + E|) \leq  v(|(m-m_0)D + (n + 1)H|) + v(|E|) =  v(|E|)$ and the theorem follows. 
\end{proof}
\renewcommand{\proofname}{Proof}

We end the section with an observation on the behavior of these loci under birational maps.

\begin{cor}
Let $f:Y \to X$ be a projective birational morphism between normal projective varieties. Then:

\begin{itemize}
\item [(i)] For every $\QQ$-Cartier $\QQ$-divisor $D$ on $X$ such that $\kappa(D)\geq 0$, we have
$$\Nna(f^*(D))=f^{-1}(\Nna(D));$$
\item [(ii)] For every pseudoeffective $\RR$-Cartier $\RR$-divisor on $X$, we have 
$$\Nnef(f^*(D))=f^{-1}(\Nnef(D));$$
\item [(iii)] If there exist effective Weil $\QQ$-divisors $\D_X$ on $X$ and $\D_Y$ on $Y$ such that $(X,\D_X)$ and $(Y,\D_Y)$ are klt pairs,
then, for every pseudoeffective $\RR$-Cartier $\RR$-divisor on $X$, we have 
$$\B_-(f^*(D))=f^{-1}(\B_-(D))$$
\end{itemize}
\end{cor}
\begin{proof}
To see (i), for every closed point $y \in Y$, let $v_y$ be a divisorial valuation such that $c_Y(v_y)=\{y\}$.
Then, by Theorem \ref{nna}, we have, 
\[ y\in f^{-1}(\Nna(D))\Leftrightarrow \{f(y)\}=c_X(v_y)\subseteq \Nna(D)\Leftrightarrow \]
\[ \Leftrightarrow v_y(\langle f^*(D) \rangle) = v_y(\langle D\rangle)>0\Leftrightarrow \{y\}=c_Y(v_y)\subseteq \Nna(f^*(D)). \]
Now (ii) can be proved exactly in the same way by using Theorem  \ref{nnef}, while (iii) follows from (ii) and \cite[Thm. 1.2]{cd}.
\end{proof}

\end{document}